\documentclass{amsart}

\usepackage{amsmath,amssymb,amsthm}
\usepackage{graphicx,tikz}
\newtheorem{theorem}{Theorem}

\newtheorem{lemma}{Lemma}

\theoremstyle{definition}

\theoremstyle{remark}

\begin{document}

\title[]{Dynamically Defined Sequences \\with small discrepancy}
\keywords{Low discrepancy sequence, energy functional.}
\subjclass[2010]{11L03, 42B05, 82C22.}

\author[]{Stefan Steinerberger}
\address{Department of Mathematics, Yale University, New Haven, CT 06511, USA}
\email{stefan.steinerberger@yale.edu}
\thanks{The author is supported by the NSF (DMS-1763179) and the Alfred P. Sloan Foundation.}

\begin{abstract} We study the problem of constructing sequences $(x_n)_{n=1}^{\infty}$ on $[0,1]$ in such a way that
$$ D_N^* = \sup_{0 \leq x \leq 1} \left| \frac{ \#\left\{1 \leq i \leq N: x_i \leq x \right\}}{N} - x \right|$$
is small. A result of Schmidt shows that for all sequences sequences $(x_n)_{n=1}^{\infty}$ on $[0,1]$ we have $D_N^* \gtrsim (\log{N}) N^{-1}$ for infinitely many $N$, several classical constructions attain this growth. We describe a type of uniformly distributed sequence that seems to be completely novel: given $\left\{x_1, \dots, x_{N-1} \right\}$, we construct $x_N$ in a greedy manner 
$$ x_N = \arg\min_{\min_k |x-x_k| \geq N^{-10}} \sum_{k=1}^{N-1}{1-\log{(2\sin{(\pi |x-x_k|)})}}.$$
We prove that $D_N^* \lesssim (\log{N}) N^{-1/2}$  and conjecture that $D_N^* \lesssim (\log{N}) N^{-1}$. Numerical examples illustrate this conjecture in a very impressive manner. We also establish a discrepancy bound $D_N^* \lesssim (\log{N})^d N^{-1/2}$ for an analogous construction in higher dimensions and conjecture it to be $D_N^* \lesssim (\log{N})^d N^{-1}$.
\end{abstract}

\maketitle

\vspace{-10pt}

\section{Introduction}
\subsection{Introduction.}   Let $(x_n)_{n=1}^{\infty}$ be a sequence on $[0,1]$ and define the star discrepancy of the first $N$ elements via
$$ D_N^* = \sup_{0 \leq x \leq 1} \left| \frac{ \#\left\{1 \leq i \leq N: x_i \leq x \right\}}{N} - x \right|.$$
van der Corput asked in 1935 whether there was a sequence for which $D_N^* \lesssim N^{-1}$. This was disproven by van Aardenne-Ehrenfest \cite{aard}, Roth \cite{roth} showed that for any sequnce
$(x_n)_{n=1}^{\infty}$ in $[0,1]$, we have $D_N^* \gtrsim \sqrt{\log{N}} N^{-1}$ for infinitely many $N$. The sharp result is due to Schmidt \cite{sch} who showed that for any sequence $(x_n)_{n=1}^{\infty}$ in $[0,1]$ there are infinitely many $N$ for which
$$D_N^* \gtrsim \frac{\log{N}}{N}.$$
Oher proofs of Schmidt's result were given by  Bejian \cite{be}, Halasz \cite{ha} and Liardet \cite{li}, the best constant is due to Larcher \& Puchhammer \cite{la}.
Several sequences attaining this growth have been constructed, we refer to the  classical textbooks by Beck \& Chen \cite{bc}, Dick \& Pillichshammer \cite{dick}, Drmota \& Tichy \cite{drmota} and Kuipers \& Niederreiter \cite{kuipers}. 
As soon as one generalizes the problem to sequences in higher dimensions $[0,1]^d$ using the notation $x=(x_1, \dots, x_d)$ and
$$ D_N^* = \sup_{x \in [0,1]^d} \left| \frac{ \#\left\{1 \leq n \leq N \big| \forall~1 \leq i \leq d~: x_{n,i} \leq x_i \right\}}{N} - \prod_{i=1}^{d}{x_i} \right|,$$
 the problem of finding sharp lower bounds on the discrepancy is open. Roth \cite{roth} proved that any sequence $\left\{x_n\right\}_{n=1}^{\infty}$ in $[0,1]^d$ has
$$ D_N^* \gtrsim \frac{(\log{N})^{\frac{d}{2}}}{N} \qquad \mbox{for infinitely many}~N.$$
An improvement by a double logarithmic factor for $d=2$ is due to Beck \cite{beck}. The best known result is due to Bilyk, Lacey, Vagharshakyan \cite{bil3, bil4} and states that for any sequence $(x_n)_{n=1}^{\infty}$ in $[0,1]^d$
$$ D_N^* \gtrsim \frac{(\log{N})^{\frac{d}{2} + \varepsilon_d}}{N} \qquad \mbox{for infinitely many}~N$$
and some $\varepsilon_d > 0$ depending only on $d$. There is no consensus on what the sharp result should be: the two main conjectures (we refer to \cite{bil}) are that for any sequence there are infinitely many $N$ such that
$$ D_N^* \gtrsim \frac{(\log{N})^{\frac{d+1}{2}}}{N} \qquad \mbox{or} \qquad  D_N^* \gtrsim \frac{(\log{N})^{d}}{N}.$$
Of course, both conjectures coincide for $d=1$. 
The first conjecture has the advantage of being structurally aligned with related conjectures in Harmonic Analysis and Probability Theory while the second conjecture has the advantage of being matched by the best known constructions. If the first conjecture were true, this would imply that in $d \geq 2$ dimensions there are sequences more regular than anything we can currently construct. Many of these classical sequences attaining $D_N^* \lesssim (\log{N})^d N^{-1}$ exploit regular structures derived from Number Theory (irrational rotations on the torus, regularity in digit expansions), so one could try to understand whether it is possible to construct
 sequences with small discrepancy using a different viewpoint.

\subsection{Results.} This paper is a companion paper to \cite{stein} where we showed that minimizing a certain functional can decrease the discrepancy of point sets. Here we show that this functional also allows us to construct uniformly distributed sequences in a way that is very different from the usual constructions.
Suppose we are 
given $\left\{x_1, \dots, x_{N-1} \right\} \subset [0,1]$, we construct $x_N$ in a greedy manner 
$$ x_N = \arg\min_{\min_k |x-x_k| \geq N^{-10}} \sum_{k=1}^{N-1}{\left(1-\log{(2\sin{(\pi |x-x_k|)})}\right)}.$$
If the minimizer is not unique, any choice is admissible. The gap condition $\min_k |x-x_k| \geq N^{-10}$ ensures that the new point $x_N$ is not extremely close to any of the existing points. We could replace it by $\min_k |x-x_k| \geq N^{-\ell}$ for any $\ell \in \mathbb{N}$ without it affecting the main result (except for constants).
One can start with any given set $\left\{x_1, \dots, x_m \right\} \subset [0,1]$ and then obtain a sequence in this greedy manner. 
\begin{theorem} We have, for any sequence thus constructed, 
$$ D_N^* \lesssim \frac{\log{N}}{\sqrt{N}},$$
where the implicit constant depends only on the initial set.
\end{theorem}
This bound in itself is not impressive (random points behave in a similar manner) but it is interesting that the outcome of such a greedy algorithm can be controlled at all. However, we believe that a much stronger statement is true: we conjecture that one can ignore the condition $\min_k |x-x_k| \geq N^{-10}$ without fundamentally altering the sequence and that one will (independently of whether one ignores $\min_k |x-x_k| \geq N^{-10}$ or not) obtain a low-discrepancy sequence.
\begin{quote}
\textbf{Conjecture 1.} For any initial set $\left\{x_1, \dots, x_m\right\} \subset [0,1]$, the greedy sequence arising out of
$$ x_N =  \arg\min_{x} \sum_{k=1}^{N-1}{\left(1-\log{(2\sin{(\pi |x-x_k|)})}\right)}$$
satisfies $D_N^* \lesssim (\log{N}) N^{-1}$. A stronger conjecture would be that the implicit constant in $D_N^* \lesssim (\log{N}) N^{-1}$ does not depend on the initial set $\left\{x_1, \dots, x_m\right\}$ as $N \rightarrow \infty$.
\end{quote}

If this statement were true, it would give rise to a large number of low-discrepancy sequences that are constructed by a technique very different from any of the usual ones. One byproduct of our argument is as follows.

\begin{theorem} Suppose we define a sequence in a greedy manner by picking $x_N$ in such a way that
$$ \sum_{n=1}^{N-1} \sum_{k=1}^{N}{ \frac{\cos(2 \pi k (x_N-x_n))}{k}} \leq 0,$$
then
$$ D_N^* \lesssim \frac{\log{N}}{\sqrt{N}}.$$
\end{theorem}
We note that it is always possible to choose such a $x_N$ since
$$ \int_{0}^{1}{\sum_{n=1}^{N-1} \sum_{k=1}^{N}{ \frac{\cos(2 \pi k (x-x_n))}{k}} dx} = 0.$$
Theorem 2 is not very deep and might be close to optimal; presumably there are various different choices of $x_N$ that are admissible and some of them might not be particularly good for the purpose of constructing low-discrepancy sequences (though, as Theorem 2 states, they cannot be arbitrarily bad either). The emphasis of our paper (as well as the numerical experimentation, see \S 1.3.) is that choosing the minimum may lead to very good behavior.
 In particular, an alternative sequence that may be interesting for further study could be
$$ x_N = \arg\min_{0 \leq x < 1} \sum_{n=1}^{N-1} \sum_{k=1}^{N}{ \frac{\cos(2 \pi k (x-x_n))}{k}}.$$
The even more general case would be 
$$ x_N = \arg\min_{0 \leq x < 1} \sum_{n=1}^{N-1}f(x-x_n) $$
for one-periodic functions $f$. There are two obvious questions: (1) are there certain functions $f$ that are particularly suited for producing regular sequences in this manner (even purely numerical results would be of interest) and (2) what can be proven about them? We emphasize that Pausinger's theorem (see \S 1.3) suggests that there might be large families of functions resulting in sequences with very good distribution properties.

\begin{figure}[h!]
\begin{minipage}[l]{.49\textwidth}
\includegraphics[width = 5.8cm]{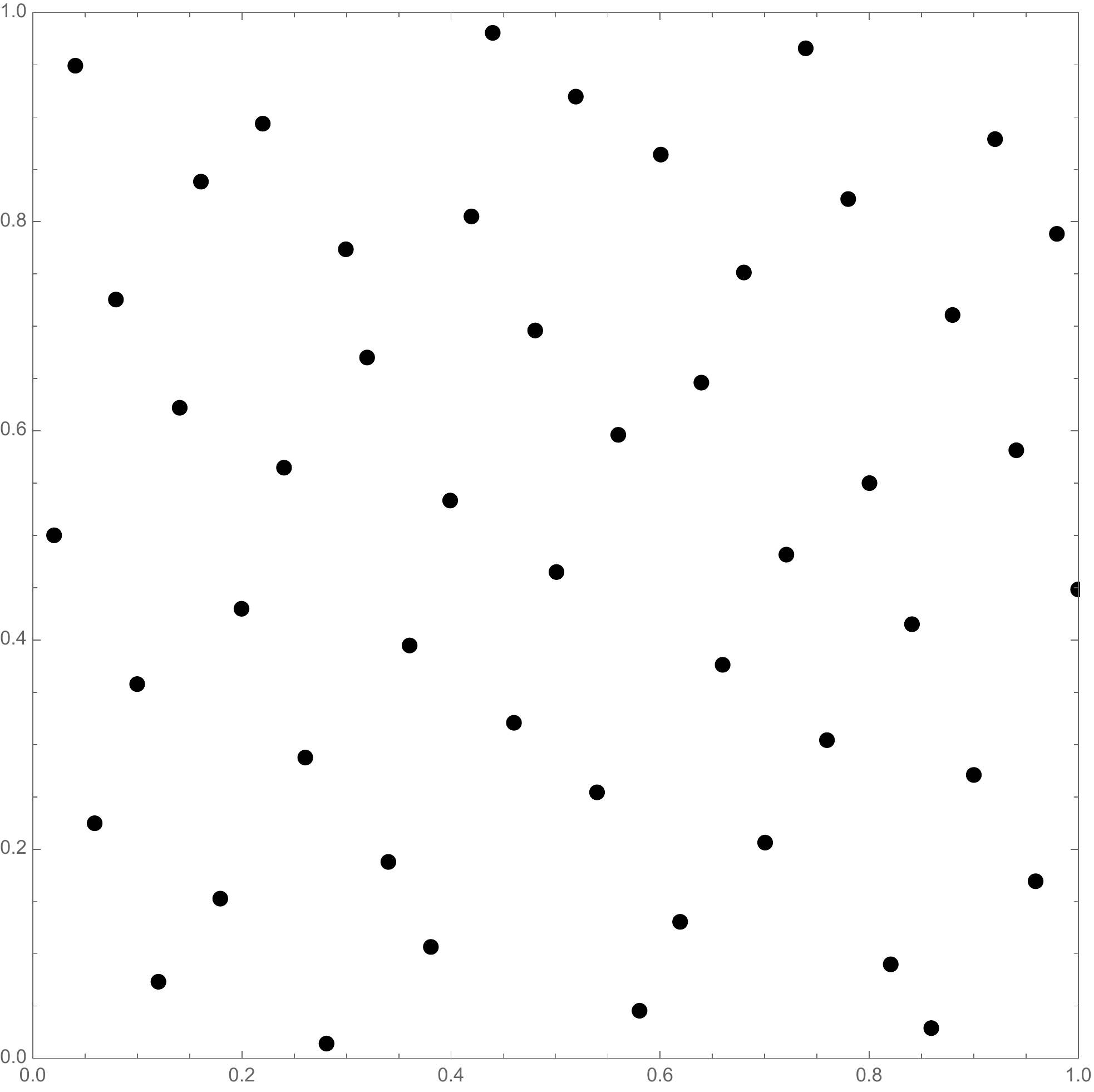} 
\end{minipage} 
\begin{minipage}[r]{.49\textwidth}
\includegraphics[width = 5.8cm]{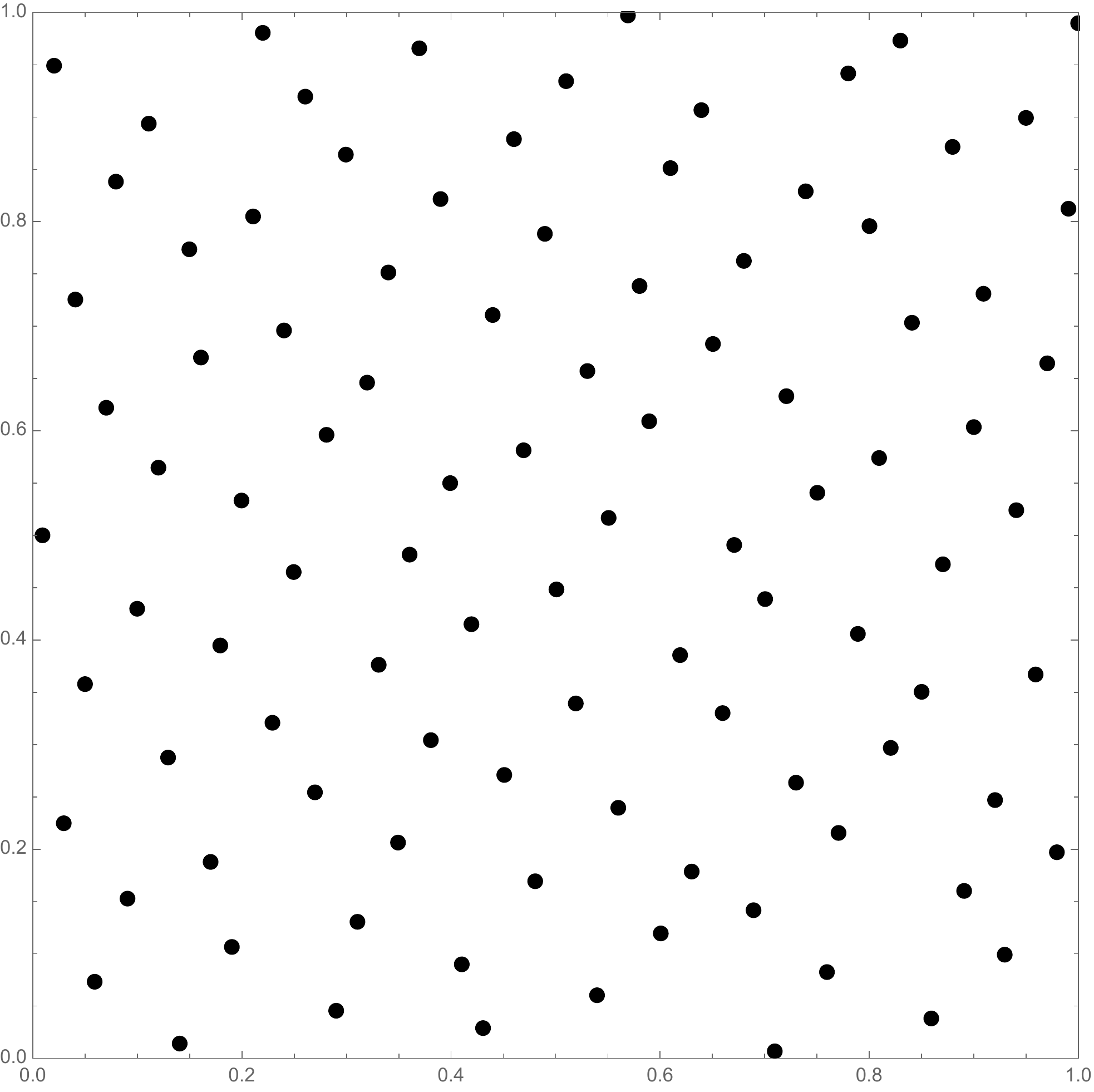} 
\end{minipage} 
\caption{$X_{50}$ (left) and $X_{100}$ (right) when starting with $\left\{0.5, 0.95\right\}$.}
\end{figure}

\subsection{Basic Numerics.} One of the main reasons why Conjecture 1 seems reasonable is that in numerical examples,
the sequence performs extraordinarily well. Given the first $N$ elements of a sequence $(x_n)_{n=1}^{N}$, we will associate to it the
 set
$$ X_N = \left\{ \left(\frac{n}{N}, x_{n}\right): 1\leq n \leq N\right\} \subset [0,1]^2$$
and use the star-discrepancy $D_N^*(X_N)$ as a sign of quality (see Fig. 1).

\begin{table}[h!]
\begin{tabular}{ l | c|  c |  c | c }
  $N$ & $D_N^*(X_N)$ & $D_N^*(\mbox{Halton}_{2,3})$ & $D_N^*(\mbox{Hammersley}_2)$ & $\mbox{Kronecker}_{\sqrt{133}}$ \\[0.05cm]
50 & 0.044 & 0.067 & 0.048  & 0.083  \\[0.05cm]
100 & 0.026 & 0.049 & 0.026  & 0.037 \\[0.05cm]
150 & 0.018  & 0.039 &  0.017 & 0.070  \\[0.05cm]
200 & 0.013   &  0.022 & 0.014   &  0.026   \\[0.05cm]
250 & 0.012   & 0.018  &  0.012  & 0.026    \\[0.05cm]
\end{tabular}
\vspace{3pt}
\caption{Discrepancy $D_N^*(X_N)$ for the sequence arising from $\left\{0.5, 0.95\right\}$ and the value of classical sets of the same size.}
\end{table}
Our sequence is actually comparable (or even superior) in quality to many of the classical constructions (see also \cite{stein}).
We compare (see Table 1) the sequence with the Halton set (using base 2 and 3), the Hammersley sequence (using base 2) and the Kronecker-type set
$$ \mbox{Kronecker}_{\sqrt{133}} = \left\{ \left(\frac{n}{N}, \frac{ \left\{ \sqrt{133} n \right\}}{N}\right): 1 \leq n \leq N\right\}.$$
The choice of $\sqrt{133}$ is more or less at random but was selected to give somewhat nice behavior (except for $N=150$, see Table 1). We observe that our sequence, starting with $\left\{0.5, 0.95\right\}$ (which was also more or less chosen at random), is comparable or superior to the other examples.
\begin{table}[h!]
\begin{tabular}{ l | c|  c |  c | c |  c | c }
  $N$ & 10 & 25 & 50 &100 & 150 & 200 \\[0.05cm]
$D_N^*(X_N)$ & 0.32  & 0.12 & 0.06 & 0.032 & 0.022 & 0.016 \\[0.05cm]
\end{tabular}
\vspace{3pt}
\caption{Discrepancy $D_N^*(X_N)$ for the sequence arising from the initial set $\left\{0.5, 0.51, 0.52, 0.53, 0.54\right\}$.}
\end{table}
This behavior seems quite robust under various initial conditions. We could try to intentionally 'break' the sequence by starting with a particularly bad initial configuration. We observe that the sequence auto-adjusts in a nice way. We illustrate this below for the sequence $(x_n)$ starting with the initial set of points $\left\{0.5, 0.51, 0.52, 0.53, 0.54\right\}$ (see Fig. 2).
 In both cases we see that the newly added points initially avoid the clustered regions and then slowly return to it (though, initially, at a lower density, see Fig. 2). We refer to Table 2 for the behavior of their star discrepancy which is initially quite large (forced by the clustered initial points) and then stabilizes very quickly.

\begin{figure}[h!]
\begin{minipage}[l]{.49\textwidth}
\includegraphics[width = 5.8cm]{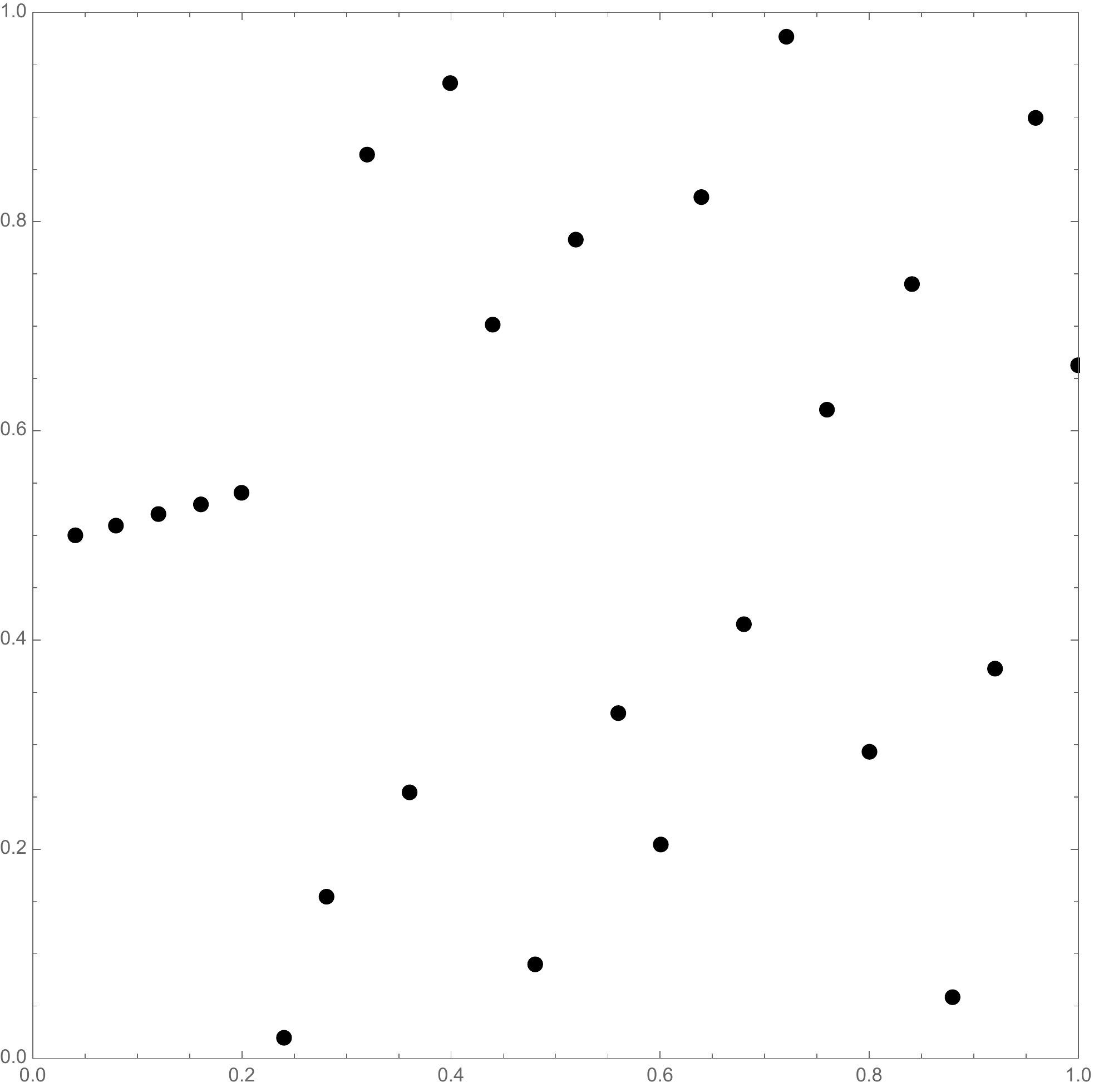} 
\end{minipage} 
\begin{minipage}[r]{.49\textwidth}
\includegraphics[width = 5.8cm]{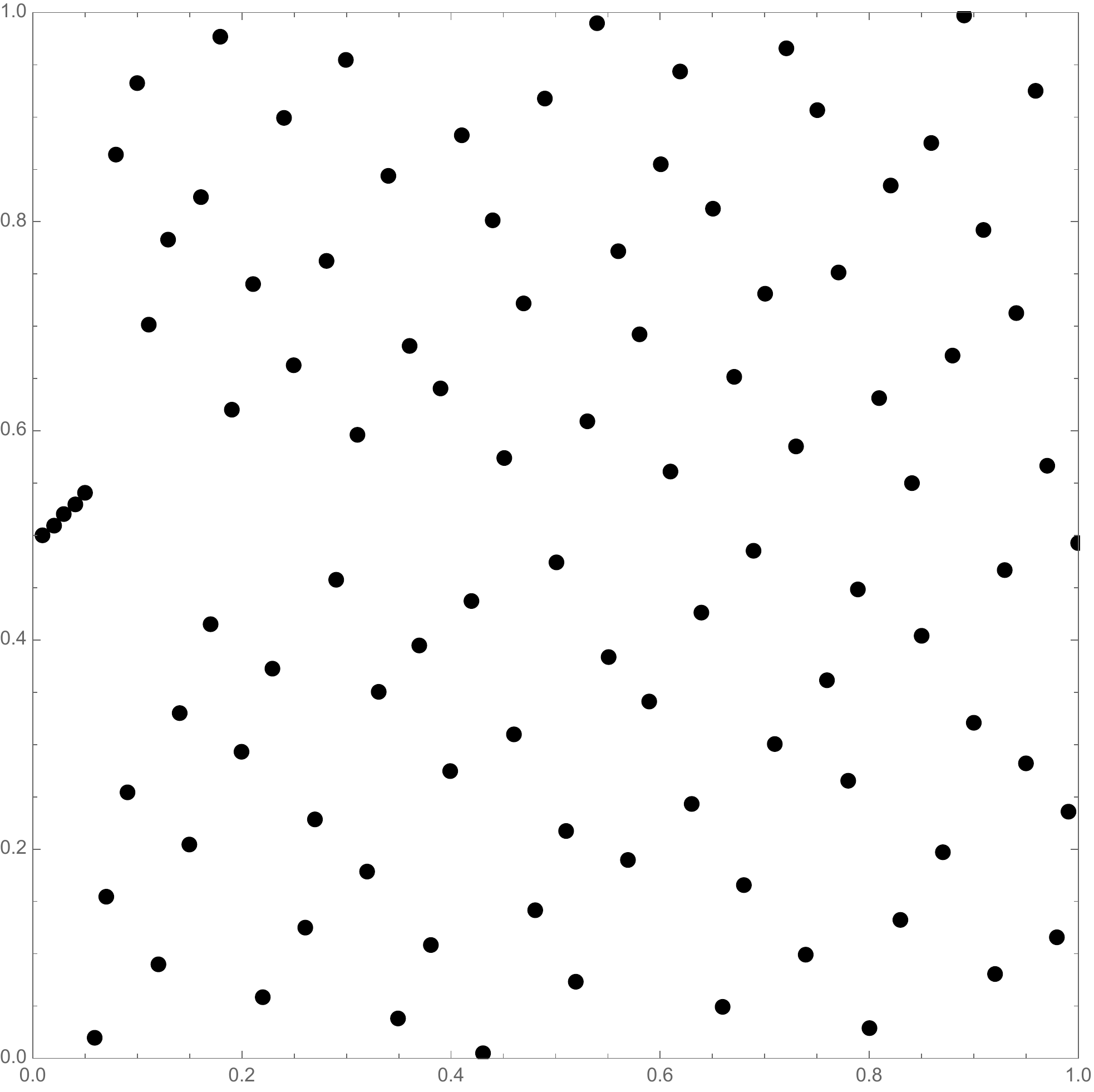} 
\end{minipage} 
\caption{$X_{25}$ (left) and $X_{100}$ (right) when starting with $\left\{0.5, 0.51, 0.52, 0.53, 0.54\right\}$. The sequence initially avoids the crowded region and then slowly returns to it}
\end{figure}

We observed numerically that certain initial conditions may be connected to variants of the van der Corput sequence in base 2. If we start with the set $\left\{0.5, 1\right\}$, then one admissible way of choosing minima leads to the sequence
$$ \frac{1}{2}, \frac{1}{1}, \frac{1}{4}, \frac{3}{4}, \frac{7}{8}, \frac{3}{8}, \frac{1}{8}, \frac{5}{8}, \frac{15}{16}, \frac{7}{16}, \frac{11}{16}, \frac{3}{16}, \frac{13}{16}, \frac{5}{16}, \frac{1}{16}, \dots$$
Can this be proven? Can admissible permutations of the van der Corput sequence, that can arise in this manner, be characterized?\\

Note added in print: this property has since been proven by Florian Pausinger \cite{pausinger} who established the following stronger result: if $f:[0,1] \rightarrow \mathbb{R}$ is symmetric, $f(1/2+x) = f(1/2-x)$, and uniformly convex, then
$$ x_N = \arg\min_x \sum_{k=1}^{N-1}{f(x-x_k)}$$
started with $x_0 = 0$ results in either the van der Corput sequence in base 2 or a permutated van der Corput sequence in base 2 (and these permutations can be precisely understood). Moreover, all these sequences have the property that $D_N^*$ has the same value for all of these sequences (i.e. depending only on $N$, not on the particular permutation). In particular, as we conjecture to be in true in general, the condition $|x-x_k| \geq N^{-10}$ is automatically enforced by the minimization problem.

\subsection{Higher dimensions.}
The construction rule for sequences in $[0,1]^d$ is slightly different: suppose we have constructed $\left\{x_1, \dots, x_{N-1}\right\} \subset [0,1]^d$, then we want the next element $x_N = (x_{N,1}, x_{N,2}, \dots, x_{N,d})$ to satisfy
$$ \sum_{n=1}^{N-1} \prod_{j=1}^{d} \left(1 + \sum_{k=1}^{N}{\frac{\cos{(2 \pi k (x_{N,j} - x_{n,j}))}}{k}}\right) \leq 1.$$
Integrating over $[0,1]^d$ shows that such a $x_N$ always exists.
\begin{theorem} Any such sequence satisfies
$$ D_N^* \lesssim \frac{(\log{N})^d}{\sqrt{N}},$$
where the implicit constant depends only on the initial set.
\end{theorem}
As in the one-dimensional case, we have the following
\begin{quote}
\textbf{Conjecture 2.} The greedy algorithm
$$ x_N =   \arg \min_x  \sum_{\ell=1}^{N-1}{ \prod_{j=1}^{d}{ \left(  1  - \log{(2 \sin{( \pi |x_{n,j} - x_{\ell,j}|)})}\right)}}.$$
leads to a sequence with $D_N^* \lesssim (\log{N})^d N^{-1}$.
\end{quote}

\subsection{Outlook.} We have introduced the sequence
$$ x_N = \arg\min_{\min_k |x-x_k| \geq N^{-10}} \sum_{k=1}^{N-1}{\left(1-\log{(2\sin{(\pi |x-x_k|)})}\right)},$$
shown that it is uniformly distributed and given some indication that it might be a low-discrepancy sequence. However, as also evidenced by the results of Pausinger \cite{pausinger}, there is
no reason to assume that there is anything particularly special about $f(x) = 1-\log{(2\sin{(\pi |x-x_k|)})}$ and similar results might be true at a much greater level of generality for large families of functions $f$. Our particular function $f$
is required to show $D_N^* \lesssim N^{-1/2} \log{N}$ since it has a Fourier-analytic connection to the Erd\H{o}s-Tur\'an inequality, it also has a natural connection to the
fractional Laplacian (we refer to the companion paper \cite{stein} for details). Nonetheless, other functions $f$ may give rise to equally good constructions and, especially in higher dimensions, it is not at all clear what function $f$ could lead to the best results.

\section{Proofs}

\subsection{A Lemma.} We start by proving a regularity statement for minimizers of the sum of logarithms. When we apply it to prove the main results, one of the relevant quantities is inside a logarithm. As a consequence, it is not tremendously important whether we prove Lemma 1 and Lemma 2 with bounds at scale $N^{-5}$ or $N^{-500}$ and thus we have not tried to optimize the arguments. 
However, a much stronger version of Lemma 1, in particular showing that the minimum is for many terms along the sequence actually at scale $ \sim -\log{N}$, could possibly improve the main result.
\begin{lemma} Let $\left\{x_1, \dots, x_N\right\} \subset [0,1]$. Then there exists $0 < x < 1$ such that
$$ \sum_{k=1}^{N}{-\log{(2\sin{(\pi |x-x_k|)})}} \lesssim -\frac{1}{N^2} \quad \mbox{and} \quad \min_{1 \leq k \leq N}{\|x-x_k\|} \gtrsim N^{-4}.$$
\end{lemma}
\begin{proof}[Proof of Lemma 1] We introduce a one-parameter family of functions for $t \geq 0$ via
$$ f_t(x) = \sum_{k \in \mathbb{Z} \atop k \neq 0}{e^{-4 \pi^2 k^2 t} \frac{ e^{2 \pi i kx}}{2|k|}}$$
and note that
$$ f_0(x) = -\log{(2\sin{(\pi |x|)})}.$$
$f_t$ is the solution of the heat equation starting with $f_0$, in particular the maximum principle for parabolic equations is telling us that for any $t > 0$
$$ \min_{x} \sum_{n=1}^{N}{f_t(x-x_n)} \geq \min_x \sum_{n=1}^{N}{f_0(x-x_n)}.$$
Moreover, by construction, for every $0<y<1$
$$ \int_{0}^{1}{f_t(x-y) dx} = \int_{0}^{1}{f_0(x-y)dx} = 0.$$
We now establish a series of bounds on $f_t$. We will work at scale $t \sim N^{-2}$ but this is not important at this point.
We first observe that
\begin{align*}
 \left\| \sum_{n=1}^{N}{f_t(x-x_n)} \right\|_{L^2}^2 &=  \left\| \sum_{n=1}^{N}{\sum_{k \in \mathbb{Z} \atop k \neq 0}{e^{-4 \pi^2 k^2 t} \frac{ e^{2 \pi i k(x-x_n)}}{2|k|}}} \right\|_{L^2}^2 \\
 &= \left\| \sum_{k \in \mathbb{Z} \atop k \neq 0}{ \frac{e^{-4 \pi^2 k^2 t}}{2|k|} e^{2\pi i k x}    \sum_{n=1}^{N}{ e^{-2 \pi i kx_n}}} \right\|_{L^2}^2\\
 &=  \sum_{k \in \mathbb{Z} \atop k \neq 0}{ \frac{e^{-8 \pi^2 k^2 t}}{4k^2} \left|  \sum_{n=1}^{N}{ e^{-2 \pi i kx_n}} \right|^2}.
 \end{align*}
 We now use a basic Lemma of Montgomery \cite{mont1} (we refer to the nice expositions in \cite[\S 5.12]{mont} and Chazelle \cite[Lemma 3.8.]{chazelle} as well as \cite{bilstein} for a recent refinement) ensuring that
 $$  \sum_{|k| \leq 100N \atop k \neq 0}{\left|  \sum_{n=1}^{N}{ e^{-2 \pi i kx_n}} \right|^2} \geq N^2$$
 and obtain
 $$ \left\| \sum_{n=1}^{N}{f_t(x-x_n)} \right\|_{L^2}^2 \gtrsim e^{-80000 \pi^2 N^2 t}.$$
 We next observe that
 $$ \left\| \frac{d}{dx} f_t \right\|_{L^{\infty}} = \left\| \frac{d}{dx}  \sum_{k \in \mathbb{Z} \atop k \neq 0}{e^{-4 \pi^2 k^2 t} \frac{ e^{2 \pi i kx}}{2|k|}} \right\|_{L^{\infty}} \lesssim \sum_{k \in \mathbb{Z} \atop k \neq 0}{ e^{-4 \pi^2 k^2 t}} \lesssim \frac{1}{\sqrt{t}}$$
 and thus
 $$ \left\| \frac{d}{dx} \sum_{n=1}^{N}{f_t(x-x_n)} \right\|_{L^{\infty}} \lesssim \frac{N}{\sqrt{t}}.$$
 Altogether, abbreviating
 $$ g(x) =  \sum_{n=1}^{N}{f_{N^{-2}}(x-x_n)},$$
 we have shown that
 $$ \int_{0}^{1}{g(x) dx} = 0, ~\|g\|_{L^2} \gtrsim 1 \quad \mbox{and} \quad \|g'\|_{L^{\infty}} \lesssim N^2.$$
 We also note that $g$ arises from the forward evolution of the heat equation and is thus smooth.
We will now use this to show that
 $$ \min_{0 < x < 1}{g(x)} \lesssim - \frac{1}{N^2}$$
 which we see as follows: clearly, from $\|g\|_{L^2} \gtrsim 1$ we observe that 
 $$\max_{0 < x < 1}{|g(x)|} \gtrsim 1.$$ If the maximum is attained at a negative value of $g(x)$, we are done. If it is attained at a positive value, then the bound on the derivative implies
 $$ \int_{0}^{1}{g^+(x) dx} \gtrsim \frac{1}{N^2}$$
 which then, with the mean 0 condition, implies 
 $$ \min_{0 < x< 1}{g(x)} \leq \int_0^1 g^{-}(x) dx = - \int_0^1 g^+(x) dx \lesssim - \frac{1}{N^2}.$$
 This implies the existence of a point $0 < x_0 < 1$ such that
 $$  \sum_{n=1}^{N}{f_{N^{-2}}(x_0-x_n)} \lesssim -\frac{1}{N^2}.$$
 This shows that the heat equation applied to $f_0$ yields a small value. We now argue that this means that the original function $f_0$ has to have a small value that is not particularly close to any of the points.
We can write (identifying the unit interval $[0,1]$ with the Torus $\mathbb{T}$)
 $$  \sum_{n=1}^{N}{f_{N^{-2}}(x_0-x_n)} = \int_{0}^{1} {\left(\sum_{n=1}^{N}{f_0(x_0 - y -x_n)} \right)  \theta_{N^{-2}}(y) dy},$$
 where 
 $$\theta_t(x) = 1 + \sum_{k \in \mathbb{Z} \atop k \neq 0}{ e^{-4 \pi^2 k^2 t} e^{2\pi i k x}}$$ 
 is the Jacobi $\theta-$function. The Jacobi $\theta-$function satisfies
 $$ \theta_t(x) \geq 0, ~ \int_{0}^{1}{\theta_t(x) dx} = 1 \quad \mbox{and} \quad \theta_t(x) \lesssim \frac{1}{\sqrt{t}}.$$
 Using the easy estimate
  $$  \sum_{n=1}^{N}{f_0(x -x_n)} \geq - N$$
  and defining
  $$ A = \left\{ x: \min_{1 \leq k \leq N} |x - x_k| \leq N^{-4}\right\} \quad \mbox{and} \quad m = \inf_{x \in A^c} \left(\sum_{n=1}^{N}{f_0(x -x_n)} \right) ,$$
  we can estimate
  \begin{align*}
   \int_{0}^{1} {\left(\sum_{n=1}^{N}{f_0(x_0 - y -x_n)} \right)  \theta_{N^{-2}}(y) dy} &\geq - |A|N + \int_{0}^{1}{ m \theta_{N^{-2}}(x) dx}\\
   &= -|A|N + m \geq - 2N^{-2} + m
   \end{align*}
   which implies $m \lesssim N^{-2}$ as desired.
  \end{proof}
  
  There is a technical step that could be slightly improved. We observe that
  \begin{align*}
   \left\| \frac{d}{dx} \sum_{n=1}^{N}{f_t(x-x_n)} \right\|_{L^{\infty}} &= \left\| \frac{d}{dx} \sum_{k \in \mathbb{Z} \atop k \neq 0}{ \frac{e^{-4 \pi^2 k^2 t}}{2|k|} e^{2\pi i k x}    \sum_{n=1}^{N}{ e^{-2 \pi i kx_n}}}  \right\|_{L^{\infty}}
   \end{align*}
The exponential cutoff localizes the sum essentially at frequency scales $\sim N$ which shows that we can expect the derivative to
be (possibly up to a logarithmic factor) at scale $\lesssim N$ as opposed to $\lesssim N^2$. However, this improvement would
have no further impact on our main result.

\subsection{An Error Bound} The second technical ingredient is straightforward.
\begin{lemma} For all $N^{-4} < x < 1-N^{-4}$ and all $M \geq N^{100}$, we have
$$\left| \sum_{k=N^{100}}^{M}{\frac{\cos{(2\pi k x)}}{k}} \right| \lesssim \frac{1}{N^{96}}.$$
\end{lemma}
\begin{proof} We use summation by parts. Summation by parts states that if $\left\{f_k\right\}, \left\{g_k\right\}$ are two sequences, then
$$ \sum_{k=m}^{n}{ f_k(g_{k+1} - g_k)} = (f_n g_{n+1} - f_m g_m) - \sum_{k=m+1}^{n}{g_k (f_k - f_{k-1})}.$$
We set 
$$ g_k = \sum_{\ell=1}^{k}{\cos{(2\pi (\ell -1) x)}} \qquad \mbox{and} \qquad f_k = \frac{1}{k}.$$
Then
\begin{align*} 
 \left| \sum_{k=N^{100}}^{M}{\frac{\cos{(2\pi k x)}}{k}} \right| \lesssim  \sup_{k}|g_k| \left(\frac{1}{N^{100}} +  \frac{1}{M} + \sum_{k=N^{100}}^{M}{\frac{1}{k^2}}\right) \lesssim  \frac{\sup_{k}|g_k|}{N^{100}}.
 \end{align*}
 It remains to estimate the supremum. We have, using $N^{-4} < x < 1-N^{-4}$,
 $$ g_k =  \sum_{\ell=1}^{k}{\cos{(2\pi (\ell -1) x)}} = \left| \Re \sum_{\ell=1}^{k}{e^{{2\pi i (\ell -1) x}}} \right| \lesssim \frac{1}{|e^{2 \pi i x} - 1|} \lesssim N^{4}.$$ 
\end{proof}

The main consequence of Lemma 1 and Lemma 2 can now be written as follows.

\begin{lemma} Let $\left\{x_1, \dots, x_{N-1}\right\} \subset [0,1]$ be arbitrary and let
$$ x_N = \arg\min_{\min_k |x-x_k| \geq N^{-10}} \sum_{k=1}^{N-1}{ \left(1-\log{(2\sin{(\pi |x-x_k|)})} \right)   }.$$
Then, for any $M \geq N^{100}$,
 $$ \sum_{n=1}^{N-1} \sum_{k=1}^{M}{ \frac{\cos(2 \pi k (x-x_n))}{k}} \leq 0.$$
\end{lemma}
\begin{proof} This follows from Lemma 1, Lemma 2 and the decomposition
$$ -\log{(2 \sin{(\pi |x|)})} = \sum_{k=1}^{\infty}{\frac{\cos{(2 \pi k x)}}{k}} = \sum_{k=1}^{M}{\frac{\cos{(2 \pi k x)}}{k}}  + \sum_{k=M + 1}^{\infty}{\frac{\cos{(2 \pi k x)}}{k}}.$$
\end{proof}

\subsection{Proof of Theorem 1 and Theorem 2}
\begin{proof}
Our derivation is motivated by the Erd\H{o}s-Turan inequality bounding the discrepancy $D_N^*$ of a set $\left\{x_1, \dots, x_N\right\} \subset [0,1]$ by
$$ D_N^* \lesssim \frac{1}{N^{100}} + \sum_{k=1}^{N^{100}}{ \frac{1}{k} \left| \frac{1}{N} \sum_{n=1}^{N}{ e^{2 \pi i k x_n} } \right|}, \qquad (\diamond)$$
where $k$ is arbitrary.
We can bound this from above by Cauchy-Schwarz
\begin{align*} \sum_{k=1}^{N^{100}}{ \frac{1}{k} \left| \frac{1}{N} \sum_{n=1}^{N}{ e^{2 \pi i k x_n} } \right|} &\leq \left(\sum_{k=1}^{N^{100}}{ \frac{1}{k}} \right)^{1/2}\left( \sum_{k=1}^{N^{100}}{  \frac{1}{k} \frac{1}{N^2} \left| \sum_{n=1}^{N}{ e^{2 \pi i k x_n} } \right|^2 }\right)^{1/2} \\
&\lesssim \sqrt{\log{N}} \left(\frac{1}{N^2}  \sum_{k=1}^{N^{100}}{ \frac{1}{k} \left| \sum_{n=1}^{N}{ e^{2 \pi i k x_n} } \right|^2 } \right)^{1/2}.
\end{align*}
 We square the second term and decouple it into diagonal and off-diagonal terms
\begin{align*}
\frac{1}{N^2} \sum_{k=1}^{N^{100}}{\left(\frac{1}{k}  \left| \sum_{n=1}^{N^{}}{ e^{2 \pi i k x_n} } \right|^2 \right)} &= \frac{1}{N^2}\sum_{k=1}^{N^{100}}{\frac{1}{k}   \sum_{n, m=1}^{N}{ e^{2 \pi i k (x_n-x_m)} } } \\
&\lesssim \frac{\log{N}}{N} + \frac{1}{N^2}\sum_{k=1}^{N^{100}}{\frac{1}{k}   \sum_{m,n =1 \atop m \neq n}^{N}{ e^{2 \pi i k (x_n-x_m)} } }
\end{align*}
Summing the points in pairs, we can simplify the double sum over $m$ and $n$ in the above expression to
$$ \sum_{m,n =1 \atop m \neq n}^{N}{ e^{2 \pi i k (x_n-x_m)} } =  2 \sum_{k=1}^{N^{100}}{\frac{1}{k}\sum_{m,n = 1 \atop m < n}^{N}{ \cos{(2 \pi k (x_m - x_n))}}}.$$
Altogether, 
\begin{align*}
\sum_{k=1}^{N^{100}}{\frac{1}{k}   \sum_{m,n =1 \atop m \neq n}^{N}{ e^{2 \pi i k (x_n-x_m)} } } = 2\sum_{m,n = 1 \atop m < n}^{N}{  \sum_{k=1}^{N^{100}}{ \frac{ \cos{(2 \pi k (x_m - x_n))}}{k} }}
 \end{align*}
Altogether, this shows that
\begin{align*}
D_N^* &\lesssim \sqrt{\log{N}} \left( \frac{\log{N}}{N} + \frac{1}{N^2}  \sum_{m,n = 1 \atop m < n}^{N}{  \sum_{k=1}^{N^{100}}{ \frac{ \cos{(2 \pi k (x_m - x_n))}}{k} }}          \right)^{1/2}\\
&= \sqrt{\log{N}} \left(\frac{\log{N}}{N} + \frac{1}{N^2} \sum_{n=2}^{N} \sum_{m=1}^{n-1}   \sum_{k=1}^{N^{100}}{\frac{ \cos{(2 \pi k (x_m - x_n))}}{k} }      \right)^{1/2}
\end{align*}
We now argue that the sum is negative because every sum (w.r.t. to $m$) is negative. Indeed, Lemma 3 implies that the choice, for every $2 \leq n \leq N$,
$$ x_n = \arg\min_{\min_k |x-x_k| \geq n^{-10}}  \sum_{m=1}^{n-1}   - \log{(2 \sin{( \pi |x_m - x|)})}$$
shows that
$$ \sum_{m=1}^{n-1}   \sum_{k=1}^{N^{100}}{\frac{ \cos{(2 \pi k (x_m - x_n))}}{k} } \leq 0.$$
This establishes the desired result. Theorem 2 follows from the same line of reasoning if we start from $(\diamond)$ with $N$ instead of $N^{100}$ and keep all the trigonometric terms (as opposed to using error bounds to move to the logarithm).
\end{proof}

We note that using this particular way of taking a limit to obtain a Fourier series was already hinted at in earlier work of the author \cite{stein1}.

\subsection{Proof of Theorem 3.} 
\begin{proof}
We use the Erd\H{o}s-Turan-Koksma inequality to bound the discrepancy of a set $\left\{x_1, \dots, x_N\right\} \subset [0,1]^d$ by

$$ D_N^* \lesssim_d  \sum_{0 < \|k\|_{\infty} \leq N}{ \frac{1}{r(2k)} \frac{1}{N} \left| \sum_{\ell=1}^{N}{e^{2\pi i \left\langle k, x_{\ell}  \right\rangle}} \right|},$$
where $r:\mathbb{Z}^d \rightarrow \mathbb{N}$ is given by
$$ r(2k) = \prod_{j=1}^{d}{\max\left\{1, 2k_j \right\}}.$$
The Cauchy-Schwarz inequality implies
$$ D_N^* \lesssim  (\log{N})^{\frac{d}{2}} \left(  \sum_{ 0 < \|k\|_{\infty} \leq N}{ \frac{1}{r(2k)} \frac{1}{N^2} \left| \sum_{\ell=1}^{N}{e^{2\pi i \left\langle k, x_{\ell}  \right\rangle}} \right|^2} \right)^{1/2}.$$
We rewrite the sum as
\begin{align*}
 \sum_{ 0 < \|k\|_{\infty} \leq N}{ \frac{1}{r(2k)} \frac{1}{N^2} \left| \sum_{\ell=1}^{N}{e^{2\pi i \left\langle k, x_{\ell}  \right\rangle}} \right|^2} =
 \frac{1}{N^2}  \sum_{ 0 < \|k\|_{\infty} \leq N}{ \frac{1}{r(2k)} \sum_{\ell, m =1}^{N}{ e^{2\pi i \left\langle k, x_{\ell} - x_m \right\rangle}}}.
 \end{align*}
 However, this sum can also be written as (after additionally summing over $k = \textbf{0}$ and then subtracting the arising value $N^2$)
 \begin{align*}
\sum_{ 0 < \|k\|_{\infty} \leq N}{ \frac{1}{r(2k)} \sum_{\ell, m =1}^{N}{ e^{2\pi i \left\langle k, x_{\ell} - x_m \right\rangle}}}
 &=  \sum_{ 0 < \|k\|_{\infty} \leq N}{ \frac{1}{r(2k)}  \sum_{\ell, m =1}^{N}{ \prod_{j=1}^{d} e^{2\pi i k_j (x_{\ell,j} - x_{m,j} )}}}\\
 &= -N^2 + \sum_{m,\ell =1}^{N}{ \prod_{j=1}^{d} \sum_{k=-N}^{N} \frac{1}{r(2k)} e^{2\pi i k_j (x_{\ell,j} - x_{m,j} )}} \\
 &= -N^2 + \sum_{m,\ell =1}^{N}{ \prod_{j=1}^{d} \left(1+ \sum_{k=1}^{N} \frac{\cos{(2\pi  k(x_{\ell,j} - x_{m,j} ))}}{ k}  \right)}.
 \end{align*}
We now separate the diagonal terms and see that
\begin{align*}
 \sum_{m,\ell =1}^{N}{ \prod_{j=1}^{d} \left(1+ \sum_{k=1}^{N} \frac{\cos{(2\pi  k(x_{\ell,j} - x_{m,j} ))}}{ k}  \right)} &\lesssim N (1 + \log{N})^d\\
 &+  \sum_{m,\ell =1 \atop m \neq \ell}^{N}{ \prod_{j=1}^{d} \left(1+ \sum_{k=1}^{N} \frac{\cos{(2\pi  k(x_{\ell,j} - x_{m,j} ))}}{ k}  \right)}.
\end{align*}
As in the proof of Theorem 1, we can reorder the sum and then use the fact that all the latter sums over $k$ negative. This shows that
 \begin{align*}
 D_N^* \lesssim  (\log{N})^{\frac{d}{2}} \left(  \sum_{ 0 < \|k\|_{\infty} \leq N}{ \frac{1}{r(k)} \frac{1}{N^2} \left| \sum_{\ell=1}^{N}{e^{2\pi i \left\langle k, x_{\ell}  \right\rangle}} \right|^2} \right)^{1/2} \lesssim \frac{(\log{N})^d}{\sqrt{N}}.
 \end{align*}
\end{proof}

\textbf{Acknowledgment.} The author is grateful to two anonymous referees whose many suggestions greatly improved the quality of the manuscript.


\begin{thebibliography}{10}


\bibitem{aard} T. van Aardenne-Ehrenfest, Proof of the Impossibility of a Just Distribution of an Infinite Sequence Over an Interval, Proc. Kon. Ned. Akad. Wetensch. 48, 3-8, 1945.


\bibitem{beck} J. Beck, A two-dimensional van Aardenne-Ehrenfest theorem in irregularities of distribution.
Compositio Math. 72 3, 269--339 (1989).

\bibitem{bc} J. Beck and W. Chen, Irregularities of Distribution, Cambridge Tracts in Mathematics (No. 89), Cambridge University Press, 1987.

\bibitem{be} R. Bejian, Minoration de la discrepance d’une suite quelconque sur T,  Acta Arith. 41 (1982), no. 2, 185--202.

\bibitem{bil}  D. Bilyk, Roth's Orthogonal Function Method in Discrepancy Theory and Some New Connections in the book "Panorama of Discrepancy Theory", Lecture Notes in Math 2107 Springer Verlag, 2014. pp. 71--158.

\bibitem{bil3} D. Bilyk and M. Lacey, On the small ball Inequality in three dimensions, Duke Math. J. 143 (2008), no. 1, 81--115.

\bibitem{bil4}  D. Bilyk, M. Lacey and A. Vagharshakyan, On the small ball inequality in all dimensions, J. Funct. Anal. 254 (2008), no. 9, 2470--2502.


\bibitem{bilstein} D. Bilyk, F. Dai and S. Steinerberger, General and Refined Montgomery Lemmata, Math. Ann., to appear.

\bibitem{chazelle} B. Chazelle, The discrepancy method. 
Randomness and complexity. Cambridge University Press, Cambridge, 2000.

\bibitem{dick} J. Dick and F. Pillichshammer, Digital nets and
  sequences. Discrepancy theory and quasi-Monte Carlo
  integration. Cambridge University Press, Cambridge, 2010.

\bibitem{drmota} M. Drmota, R. Tichy,  Sequences, discrepancies and applications. Lecture Notes in Mathematics, 1651. Springer-Verlag, Berlin, 1997.

\bibitem{ha} G. Halasz, On Roth’s method in the theory of irregularities of point distributions, in: Recent Progress in Analytic Number Theory, 2 (Durham, 1979), Academic Press, London, 1981, pp. 79--94.

\bibitem{kuipers} L. Kuipers and H. Niederreiter, Uniform distribution of sequences. Pure and Applied Mathematics. Wiley-Interscience, New York-London-Sydney, 1974.

\bibitem{la} G. Larcher and F. Puchhammer, An improved bound for the star discrepancy of
sequences in the unit interval, Uniform Distribution Theory 11, no. 1, 1--14 (2016).

\bibitem{li} P. Liardet, Discrepance sur le cercle., Primaths. I, Univ. Marseille, 1979

\bibitem{mont1} H. Montgomery, Irregularities of distribution by means of power sums, Congress of Number Theory (Zarautz, 1984), Universidad del Pa{\i}s Vasco Bilbao, 1989, 11-27.
\bibitem{mont} H. Montgomery, Ten Lectures at the Interface of Harmonic Analysis and Number Theory, American Mathematical Society, 1994.

\bibitem{pausinger} F. Pausinger, Greedy energy minimization can count in binary: point charges and the van der Corput sequence, arXiv:1905.09641

\bibitem{roth}  K. F. Roth,  On irregularities of distribution. Mathematika 1, 73--79 (1954).

\bibitem{sch} W. Schmidt, Irregularities of distribution. VII.  Acta Arith. 21 (1972), 45--50. 


\bibitem{stein1} S. Steinerberger, Poissonian Pair Correlation and Discrepancy, Indag. Math. 29, 1167--1178  (2018).

\bibitem{stein} S. Steinerberger, A Nonlocal Functional promoting Low-Discrepancy Point Sets, Journal of Complexity, accepted.


\end{thebibliography}
\end{document}